\documentclass[
a4paper,
twoside,
]{amsart}
\usepackage{amssymb} 
\usepackage[all]{xy} 
\usepackage{color} 
\usepackage{lastpage} 
\usepackage[pdfstartview=FitH]{hyperref} 

\newtheorem{thm}{Theorem}[section]

\newtheorem{cor}[thm]{Corollary}
\newtheorem{lem}[thm]{Lemma}

\newtheorem{prop}[thm]{Proposition}

\theoremstyle{definition}
\newtheorem{defn}[thm]{Definition}

\theoremstyle{remark}
\newtheorem{rem}[thm]{Remark}

\numberwithin{equation}{section}
\newcommand*{\cat}[1]{\mathbf{#1}} 
\newcommand*{\mto}{\rightarrow} 
\newcommand*{\del}{\partial} 
\newcommand*{\comp}{\circ} 
\newcommand*{\Int}{\mathbb{Z}} 
\newcommand*{\Rat}{\mathbb{Q}} 
\newcommand*{\cmplx}[1]{{#1}^\bullet} 
\newcommand*{\tensor}{\otimes} 
\newcommand*{\ctensor}{\hat{\otimes}} 
\newcommand*{\isomorph}{\cong} 
\newcommand*{\op}{\mathrm{op}} 
\newcommand{\Kspace}{\mathbb{K}} 

\newcommand{\id}{\mathrm{id}} 
\newcommand{\SQab}{\mathcal{SQ}_1^{\mathrm{ab}}} 

\DeclareMathOperator{\HF}{H} 
\DeclareMathOperator{\Ann}{Ann} 
\DeclareMathOperator{\im}{im} 
\DeclareMathOperator{\KTh}{K} 

\newdir{ >}{{}*!/-5pt/@{>}} 
\newdir{O}{{}*-\cir<2.5pt>{}} 
\xyoption{2cell} 
\UseTwocells %
\xyoption{poly}

\allowdisplaybreaks[1]

\begin{document}


\title[Localisation Sequence]{On a Localisation Sequence for the $\KTh$-Theory of Skew Power Series Rings}
\author{Malte Witte}%

\address{Malte Witte\newline Ruprecht-Karls-Universit\"at Heidelberg\newline
Mathematisches Institut\newline
Im Neuenheimer Feld 288\newline
D-69120 Heidelberg }%
\email{witte@mathi.uni-heidelberg.de}%


\date{\today}%

\begin{abstract}
Let $B=A[[t;\sigma,\delta]]$ be a skew power series ring such that $\sigma$ is given by an inner automorphism of $B$. We show that a certain Waldhausen localisation sequence involving the $\KTh$-theory of $B$ splits into short split exact sequences. In the case that $A$ is noetherian we show that this sequence is given by the localisation sequence for a left denominator set $S$ in $B$. If $B=\Int_p[[G]]$ happens to be the Iwasawa algebra of a $p$-adic Lie group $G\isomorph H\rtimes \Int_p$, this set $S$ is Venjakob's canonical Ore set. In particular, our result implies that
$$
0\mto \KTh_{n+1}(\Int_p[[G]])\mto \KTh_{n+1}(\Int_p[[G]]_S)\mto \KTh_n(\Int_p[[G]],\Int_p[[G]]_S)\mto 0
$$
is split exact for each $n\geq 0$. We also prove the corresponding result for the localisation of $\Int_p[[G]][\frac{1}{p}]$ with respect to the Ore set $S^*$. Both sequences play a major role in non-commutative Iwasawa theory.
\end{abstract}

\maketitle

\section*{Introduction}

Let $B=A[[t;\sigma,\delta]]$ be a skew power series rings and assume that $\sigma$ extends to an inner automorphism of $B$. In this article, we exploit an observation first made by Schneider and Venjakob \cite{SchnVen:DimThSPSR}, namely the existence of a short exact sequence
$$
(*)\qquad 0\mto B\ctensor_A M\mto B\ctensor_A M\mto M\mto 0
$$
for a pseudocompact $B$-module $M$. (In \emph{loc. cit.}, $M$ was also assumed to be finitely generated over $A$, but we will show below that this is not necessary.) This was used --- first in the context of non-commutative Iwasawa theory by Burns \cite{Burns:AlgebraicLfunctions}, then in general by Schneider and Venjakob \cite{SchnVen:LocCompSPSR} --- to construct a natural splitting of the boundary map
$$
\KTh_{1}(B_S)\xrightarrow{\del_0} \KTh_0(B,B_S)
$$
in the $\KTh$-theory localisation sequence of a regular noetherian skew power series ring $B$ and its localisation $B_S$ with respect to a certain left denominator set $S$.

Generalising this result, we show that by an application of Waldhausen's additivity theorem \cite{Wal:AlgKTheo} to the above exact sequence and by a well-known result in homotopy theory one obtains a natural splitting of the boundary map $\del_n$ for every $n$. Using Waldhausen's construction of $\KTh$-theory for arbitrary Waldhausen categories also helps us to extend this result to non-regular skew power series and even to situations where no suitable left denominator set $S$ is known to exist.

The main application of our result lies in the area of non-commutative Iwasawa theory. Let $H$ be a $p$-adic Lie group and $G$ a semi-direct product of $H$ and $\Int_p$. Fix a topological generator $\gamma$ of $\Int_p$. Then $B=\Int_p[[G]]$ is a skew power series ring over $A=\Int_p[[H]]$ with $\sigma$ given by conjugation with $\gamma$, $t=\gamma-1$ and $\delta=\sigma-\id$. The set $S$ is Venjakob's canonical Ore set and the sequence
$$
\KTh_1(\Int_p[[G]])\mto\KTh_1(\Int_p[[G]]_S)\xrightarrow{\del_0}\KTh_0(\Int_p[[G]],\Int_p[[G]]_S)\mto 0
$$
is used to formulate the main conjecture of non-commutative Iwasawa theory (see \cite{CFKSV}, \cite{RW:EquivIwaTh1}, \cite{FK:CNCIT}, and subsequent articles). By our result, the leftmost map in this sequence is always injective and the splitting of $\del_0$ does exist even if $G$ has $p$-torsion elements. It was part of the proof of the non-commutative main conjecture for totally real fields \cite{Kakde2} to verify this fact for Galois groups of admissible $p$-adic Lie extensions.

The authors of \cite{CFKSV} also use the localisation sequence for the left denominator set
$$
S^*=\bigcup_{n\geq 0} p^nS.
$$
Using the same tools as above, we can show that the boundary morphism
$$
\KTh_{n+1}(\Int_p[[G]]_{S^*})\mto\KTh_n(\Int_p[[G]][\frac{1}{p}],\Int_p[[G]]_{S^*})
$$
splits as well for every $n\geq 0$.

In \cite{SchnVen:DimThSPSR} the coefficient ring $A$ of the skew power series ring $B$ is always assumed to be noetherian. This restriction excludes important examples such as the completed group ring of the absolute Galois group of a global field. Therefore, we will work in a much larger generality by only assuming that $A$ is pseudocompact as $A$-$A$-bimodule. For the construction of the left denominator set $S$ we cannot dispose of the noetherian condition, but we can considerably weaken the extra conditions required in \cite{SchnVen:LocCompSPSR}.

In those cases where a suitable $S$ is not known to exist, we obtain a split exact sequence
$$
0\mto\KTh_{n+1}(B)\mto\KTh_{n+1}(\cat{TSP}_A(B))\mto\KTh_{n}(\cat{TSP}^A(B))\mto 0
$$
for certain Waldhausen categories $\cat{TSP}_A(B)$ and $\cat{TSP}^A(B)$. For applications in Iwasawa theory, this is still a perfectly sensible construction (see e.\,g.\ \cite{Witte:MCVarFF}).

The article is structured as follows. In Section~\ref{sec:pseudocompact modules} we recall properties of pseudocompact modules and of the completed tensor product. For the latter, some extra care is needed because we do not assume that our rings are pseudocompact as modules over a commutative subring as it is done in \cite{Brumer:PseudocompactAlgebras}. In Section~\ref{sec:localisation sequence} we introduce skew power series ring, prove the existence of the short exact sequence $(*)$ and then the splitting of the localisation sequence in terms of the Waldhausen categories $\cat{TSP}_A(B)$ and $\cat{TSP}^A(B)$. We then introduce the set $S$ and prove the Ore condition in the noetherian case. The $\KTh$-groups of the Waldhausen categories may then be replaced by $\KTh_n(B_S)$ and $\KTh_n(B,B_S)$. In Section~\ref{sec:Applications} we consider the special case of completed group rings. In this section, we also prove our result for the localisation with respect to $S_*$.

All rings will be associative with identity. The opposite ring of a ring $R$ will be denoted by $R^{\op}$. An $R$-module is always understood to be a left unitary $R$-module.

The author would like to thank A.~Schmidt and O.~Venjakob for their advice and support.

\section{Preliminaries on Pseudocompact Modules}\label{sec:pseudocompact modules}

Pseudocompact modules have been considered in \cite{Gab:CatAb} and \cite{Brumer:PseudocompactAlgebras}. At some points, it is also convenient to work in the larger category of pro-discrete modules.

\subsection{First Properties}\label{ss:first properties}

\begin{defn}\label{defn:pseudocompact}
Let $A$ be a topological ring and let $M$ be a topological $A$-module.
\begin{enumerate}
    \item The module $M$ is \emph{pro-discrete} if
    $$
        M=\varprojlim_{i\in I}M_i
    $$
    is the limit of a filtered system of discrete $A$-modules $M_i$.
    \item The module $M$ is \emph{pseudocompact} if in addition the $A$-modules $M_i$ are of finite length.
\end{enumerate}
\end{defn}

For any topological $A$-module $M$, we will denote by $\mathcal{U}_M$ the directed set of open submodules of $M$.

\begin{prop}\label{prop:char of LC PC}
Let $M$ be a topological $A$-module.
\begin{enumerate}
    \item $M$ is pro-discrete if and only if for each $U\in\mathcal{U}_M$, $M/U$ is a discrete $A$-module and
    $$
    M\mto\varprojlim_{U\in\mathcal{U}_M} M/U
    $$
    is a topological isomorphism.
    \item $M$ is pseudocompact if and only if it is pro-discrete and for each $U\in\mathcal{U}_M$, $M/U$ has finite length.
\end{enumerate}
\end{prop}
\begin{proof}
Clearly the given conditions are sufficient. To prove necessity, consider
$$
M=\varprojlim_{i\in I} M_i
$$
with $M_i$ discrete and let $\phi_i\colon M\mto M_i$ denote the structure morphisms. The $\ker \phi_i$ form a cofinal system in $\mathcal{U}_M$ and for each $i$, the induced homomorphism $M/\ker \phi_i\mto M_i$ is injective and $M/\ker \phi_i$ is a discrete $A$-module. It has finite length if $M_i$ does. This is then also true for all $U\in\mathcal{U}_M$ containing $\ker \phi_i$. Since the composition
$$
M\xrightarrow{\alpha}\varprojlim_{U\in\mathcal{U}_M}M/U=\varprojlim_{i\in I}M/\ker \phi_i \xrightarrow{\beta} \varprojlim_{i\in I}M_i
$$
is a topological isomorphism and $\beta$ is injective, $\alpha$ is a topological isomorphism as well.
\end{proof}

\begin{rem}
The proposition shows that a topological $A$-module $M$ is pro-discrete precisely if it is linearly topologised, separated, and complete.
\end{rem}

We will denote the category of pro-discrete $A$-modules (with continuous homomorphisms) by $\cat{PD}(A)$, the category of pseudocompact $A$-modules by $\cat{PC}(A)$. The category $\cat{PD}(A)$ is additive and
has arbitrary limits, but is not abelian. The full additive subcategory $\cat{PC}(A)$ is much better behaved.

\begin{prop}\label{prop:ab cat PC}
The category $\cat{PC}(A)$ is abelian and has arbitrary exact products as well as exact filtered limits.
\end{prop}
\begin{proof}
See  \cite[\S 5, Thm. 3]{Gab:CatAb} and note that the proof does not need $A$ to be pseudocompact as left $A$-module.
\end{proof}

\subsection{The Completed Tensor Product}\label{ss:completed tensor product}

\begin{defn}\label{defn:completed tensor}
Let $A$ be a topological ring. For $M$ in $\cat{PD}(A^{\op})$ and $N$ in $\cat{PD}(A)$, we define the \emph{completed tensor product}
$$
M\ctensor_A N=\varprojlim_{U\in\mathcal{U}_M}\varprojlim_{V\in\mathcal{U}_N}M/U\tensor_A N/V
$$
where we give $M/U\tensor_A N/V$ the discrete topology.
\end{defn}

\begin{prop}\label{prop:univ prop tensor}
The completed tensor product $M\ctensor_A N$ satisfies the following universal property.
\begin{enumerate}
\item $M\ctensor_A N$ is pro-discrete as $\Int$-module (with respect to the discrete topology on $\Int$) and the canonical map
    $$
    M\times N\mto M\ctensor_A N,\qquad (m,n)\mapsto m\ctensor n,
    $$
    is continuous and bilinear.
\item For any pro-discrete $\Int$-module $T$ and any continuous and bilinear map $f\colon M\times N\mto T$, there exists a unique continuous homomorphism
$$
\bar{f}\colon M\ctensor_A N\mto T
$$
with $\bar{f}(m\ctensor n)=f(m,n)$.
\end{enumerate}
\end{prop}
\begin{proof}
This follows easily from the universal property of the usual tensor product, see \cite[Sect.~2]{Brumer:PseudocompactAlgebras}.
\end{proof}

\begin{prop}\label{prop:commuting with filtered limits}
Let $M$ be a pro-discrete $A^{\op}$-module, let $N$ be the limit of a filtered system $(N_i)_{i\in I}$ of pro-discrete $A$-modules, and assume that all structure maps $\phi_i\colon N\mto N_i$ are surjective. Then
$$
M\ctensor_A \varprojlim_{i\in I}N_i\mto \varprojlim_{i\in I} M\ctensor_A N_i
$$
is a topological isomorphism. In particular, the completed tensor product commutes with arbitrary products of pro-discrete $A$-modules and with filtered limits of pseudocompact $A$-modules.
\end{prop}
\begin{proof}
For each $U\in \mathcal{U}_{N_i}$, the map $\phi_{i,U}\colon N\mto N_i/U$ is surjective and the set
$\{\ker \phi_{i,U}\;\mid\;\text{$i\in I$, $U\in\mathcal{U}_{N_i}$}\}$
is a cofinal subset of $\mathcal{U}_N$. Hence,
$$
M\ctensor_A N=\varprojlim_{i\in I}\varprojlim_{V\in\mathcal{U}_{N_i}}\varprojlim_{U\in\mathcal{U}_M}M/U\tensor_A N_i/V=\varprojlim_{i\in I}M\ctensor_A N_i
$$
as claimed.

Finally, note that every direct product over a set $I$ may be rewritten as the limit of the products over finite subsets of $I$ and that by Prop.~\ref{prop:ab cat PC} the maps $\phi_i$ are surjective if all $N_i$ are pseudocompact.
\end{proof}

\begin{defn}\label{defn:bimodules}
Let $A$ and $B$ be topological rings and $M$ a topological $A$-$B$-bimodule, i.\,e.\ a topological abelian group with compatible topological $A$-module and $B^{\op}$-module structures. We call $M$
\begin{enumerate}
\item \emph{pro-discrete} if $M$ is pro-discrete as $A$-module and there exists a fundamental system of neighbourhoods of $0$ consisting of open $B^\op$-submodules,
\item \emph{pseudocompact} if $M$ is pseudocompact both as $A$-module and as $B^{\op}$-module.
\end{enumerate}
\end{defn}

Note that a pro-discrete $A$-$B$-bimodule $M$ is automatically pro-discrete as $B^{\op}$-module. Indeed, any open $A$-submodule $U$ of $M$ contains an open $B^{\op}$-submodule $V$. Hence, it also contains
$$
\overline{\sum_{a\in A}aV},
$$
which is an open $A$-$B$-subbimodule of $M$. In particular,
$$
M\isomorph\varprojlim_{U\in\mathcal{U}'_M}M/U
$$
where $\mathcal{U}'_M$ denotes the filtered set of open subbimodules of $M$. The bimodule $M$ is pseudocompact precisely if it is pro-discrete and for each $U\in\mathcal{U}'_M$, $M/U$ is of finite length as $A$-module and noetherian as $B^{\op}$-module \cite[Thm.~8.12]{GW:NoncommNoethRings}.

\begin{prop}\label{prop:existence base change}
Let $A$ and $B$ be topological rings and $M$ a pro-discrete $A$-$B$-bimodule.
\begin{enumerate}
\item If $N$ is a pro-discrete $B$-module, then $M\ctensor_B N$ is a pro-discrete $A$-module.
\item If $M$ is pseudocompact as $A$-module and $N$ is a pseudocompact $B$-module, then $M\ctensor_B N$ is a pseudocompact $A$-module.
\end{enumerate}
\end{prop}
\begin{proof}
One reduces to the case that $M$ and $N$ carry the discrete topology. For the first assertion, one then needs to verify that $M\tensor_B N$ with its discrete topology is a topological $A$-module. This is the case precisely if the annihilator of each element is open in $A$.  The relations
$$
\Ann_A m\subset\Ann_A m\tensor n \qquad\text{and}\qquad \Ann_A x\cap \Ann_A y\subset \Ann_A x+y
$$
imply that this condition is satisfied.

Further, if $M$ is of finite length as $A$-module and $N$ is a finitely generated $B$-module, then $M\tensor_B N$ is again of finite length. Hence, the second assertion follows.
\end{proof}

The completed tensor product is in general neither left exact nor right exact, but we obtain right exactness in the following situation.

\begin{prop}\label{prop:tensor right exact}
Let $A$ and $B$ be topological rings and $M$ be a pro-discrete $A$-$B$-bimodule which is pseudocompact as $A$-module. Then the functor
$$
\cat{PC}(B)\mto\cat{PC}(A),\qquad N\mapsto M\ctensor_B N,
$$
is right exact.
\end{prop}
\begin{proof}
This follows from the exactness of filtered limits (Prop.~\ref{prop:ab cat PC}), the right exactness of the usual tensor product, and Prop.~\ref{prop:commuting with filtered limits}.
\end{proof}

\subsection{Pseudocompact Rings}\label{ss:pseudocompact rings}

\begin{defn}\label{defn:pseudocompact rings}
The topological ring $A$ is called
\begin{enumerate}
\item \emph{left pseudocompact} if it is pseudocompact as a topological (left) $A$-module,
\item\emph{(left and right) pseudocompact} if it is pseudocompact as a topological $A$-$A$-bimodule.
\end{enumerate}
\end{defn}

Our terminology agrees with the notion of pseudocompact rings in \cite{Brumer:PseudocompactAlgebras}. In \cite{Gab:CatAb}, the term ``pseudocompact ring'' refers to left pseudocompact rings in our terminology. Nevertheless, the following results remain true.

\begin{prop}\label{prop:free modules}
If $A$ is a left pseudocompact ring, then every pseudocompact module is quotient of a direct product of copies of $A$. Furthermore, the projectives in $\cat{PC}(A)$ are the direct summands of arbitrary direct products of copies of $A$.
\end{prop}
\begin{proof}
See \cite[Cor.~1.3, Lemma~1.6]{Brumer:PseudocompactAlgebras}.
\end{proof}

\begin{prop}\label{prop:finitely presented modules}
Let $A$ be a left pseudocompact ring. The forgetful functor from the category of finitely presented, pseudocompact $A$-modules to the category of finitely presented, abstract $A$-modules is an equivalence of categories. Moreover, any abstract homomorphism from a finitely presented  $A$-module to a pseudocompact $A$-module is automatically continuous.
\end{prop}
\begin{proof}
This follows as \cite[Prop.~5.2.22]{NSW:CohomNumFields}.
\end{proof}

For the following result, we must restrict to (left and right) pseudocompact rings.

\begin{prop}\label{prop:comparison of tensor prods}
Let $A$ be a pseudocompact ring.
\begin{enumerate}
\item For any pseudocompact $A$-module $N$, the canonical homomorphism
$$
N\mto A\ctensor_A N
$$
is an isomorphism.
\item If $B$ is a topological ring, $M$ a pseudocompact $B$-$A$-bimodule, and $N$ a finitely presented $A$-module, then the canonical homomorphism
$$
M\tensor_A N\mto M\ctensor_A N
$$
is an isomorphism.
\end{enumerate}
\end{prop}
\begin{proof}
For the first assertion we choose a presentation
$$
\prod_{i\in I}A\mto\prod_{j\in J}A\mto N\mto 0
$$
using Prop.~\ref{prop:free modules}. Then we use the right-exactness of the completed tensor product (Prop.~\ref{prop:tensor right exact}) and that the completed tensor product commutes with arbitrary products (Prop.~\ref{prop:commuting with filtered limits}) to reduce to the case $N=A$. In this case, the statement follows immediately from the definition.

The second assertion follows by choosing a finite presentation of $N$ and applying the first assertion to the pseudocompact $A^{\op}$-module $M$. Since $M$ is also pseudocompact as $B$-module, the completed tensor product is again right exact such that
$$
M\tensor_A N\mto M\ctensor_A N
$$
is indeed an isomorphism.
\end{proof}

\begin{rem}
Note that for the second assertion of the proposition, it does not suffice to assume that $N$ is pseudocompact and finitely generated. The proof of \cite[Lemma 2.1.(ii)]{Brumer:PseudocompactAlgebras} is erroneous since the kernel of
$$
M\tensor_A A^n\mto M\tensor_A N
$$
is not necessarily closed.
\end{rem}

\section{The Localisation Sequence}\label{sec:localisation sequence}

\subsection{Skew Power Series Rings}

Let $A$ be a pseudocompact ring with a continuous ring automorphism $\sigma$. Recall that a $\sigma$-derivation on $A$ is a continuous map
$
\delta\colon A\mto A
$ satisfying
$$
\delta(ab)=\sigma(a)\delta(b)+\delta(a)b
$$
for all $a,b\in A$.

\begin{defn}
Let $A$ be a (left and right) pseudocompact ring with a continuous automorphism $\sigma$ and a topologically nilpotent $\sigma$-derivation $\delta$.
Assume that $B$ is a left pseudocompact ring with an element $t\in B$ such that
\begin{enumerate}
\item $B$ contains $A$ as a closed subring,
\item $B$ is pseudocompact as $A^{\op}$-module,
\item as topological $A$-module, $B$ is topologically freely generated by the powers of $t$:
$$B=\prod_{n=0}^{\infty}At^n,$$
\item for any $a\in A$,
$$ta=\sigma(a)t+\delta(a).$$
\end{enumerate}
Then $B=A[[t;\sigma,\delta]]$ will be called a \emph{(left) skew power series ring} over $A$.
\end{defn}

\begin{rem}
A skew power series ring does not exist for every choice of $\sigma$ and $\delta$. If it does exist, it is not clear that it is also a right skew power series ring, i.\,e.\ topologically free over the powers of $t$ as a topological $A^{\op}$-module. A condition that is clearly necessary for the existence and pseudocompactness is that $\delta$ is topologically nilpotent, i.\,e.\ that for each open ideal $I$ of $A$, there exists a positive integer $n$ such that $\delta^n(A)\subset I$. If $A$ is noetherian, it is shown in \cite{SchnVen:DimThSPSR} that a sufficient condition for the existence of a left and right skew power series ring is that $\delta$ is $\sigma$-nilpotent. Analysing the multiplication formula (4) in \cite[Sect.~1]{SchnVen:DimThSPSR} one sees that a sufficient condition for the existence as left skew power series ring is that $\delta$ is topologically nilpotent and that $A$ admits a fundamental system of open neighbourhoods of $0$ consisting of two-sided ideals stable under $\sigma$ and $\delta$. More general definitions of skew power series rings are considered in \cite[\S 7.4.2]{Rueschoff:diplom}.
\end{rem}

\begin{rem}\label{rem:change of variable}
If $B=A[[t;\sigma,\delta]]$ is a left skew power series ring, $a\in A^{\times}$ and $b\in A$, then we may consider $B$ also as skew power series ring $B=A[[t';\sigma',\delta']]$ in $t'=at+b$ with the $A$-automorphism
$$
\sigma'\colon A\mto A,\qquad x\mapsto a\sigma(x)a^{-1},
$$
and the $\sigma'$-derivation
$$
\delta'\colon A\mto A, \qquad x\mapsto a\delta(x)+bx-\sigma'(x)b.
$$
For this, we note that
$$
B\mto B,\qquad \sum_{n=0}^{\infty} f_nt^n\mapsto \sum_{n=0}^{\infty} f_n (t')^n,
$$
is an isomorphism of pseudocompact left $A$-modules and hence, $B$ is also topologically freely generated by the powers of $t'$.
\end{rem}

We fix a skew power series ring $B=A[[t;\sigma,\delta]]$ and let $B^\sigma$ denote the pseudocompact $B$-$A$-bimodule with the $A^\op$-structure given by
$$
a\cdot b=b\sigma(a)
$$
for $a\in A^\op$, $b\in B$.

The following proposition extends \cite[Prop.~2.2]{SchnVen:DimThSPSR}. An important difference is that it applies not only to $B$-modules finitely generated as $A$-modules, but to all pseudocompact $B$-modules.

\begin{prop}\label{prop:the exact sequence}
Let $B$ be a skew power series ring over $A$. For any pseudocompact $B$-module $M$, there exists an exact sequence
of pseudocompact $B$-modules
$$
0\mto B^\sigma\ctensor_A M\xrightarrow{\kappa} B\ctensor_A M\xrightarrow{\mu} M\mto 0
$$
with
$$
\kappa(b\ctensor m)=bt\ctensor m-b\ctensor tm,\qquad \mu(b\ctensor m)=bm.
$$
\end{prop}
\begin{proof}
We closely follow  the argumentation of \cite[proof of Prop.~2.2]{SchnVen:DimThSPSR}, replacing generators by topological generators. Only the proof of the injectivity of $\kappa$ is different.

From Prop.~\ref{prop:free modules} and the definition of skew power series rings one deduces that $M$ is pseudocompact as $A$-module. Prop.~\ref{prop:existence base change} then implies that $B\ctensor_A M$ and $B^\sigma\ctensor_A M$
are pseudocompact as $B$-modules. Using the universal property of the completed tensor product (Prop.~\ref{prop:univ prop tensor}) one  shows that $\kappa$ and $\mu$ are well defined and functorial
continuous $B$-homomorphisms. Clearly, $\mu\comp\kappa=0$. The map $\mu$ is always surjective, since
$$
M\mto B\ctensor_A M,\qquad m\mapsto 1\ctensor m,
$$
is a continuous $A$-linear splitting by Prop.~\ref{prop:univ prop tensor}. One then observes that $\ker \mu$  is topologically generated as pseudocompact $A$-module by
$$
t^i\ctensor m-1\ctensor t^im=\sum_{k=0}^{i-1} \kappa(t^{i-1-k}\ctensor t^{k}m)
$$
where $i$ is an positive integer and $m$ runs through a topological generating system of the pseudocompact $A$-module $M$. Hence, $\ker \mu=\im \kappa$.

We will now consider the special case $M=B$. The canonical isomorphism of left pseudocompact $A$-modules
$$
B\ctensor_A B=B\ctensor_A \prod_{i=0}^{\infty} A=\prod_{i=0}^{\infty}B=\prod_{i=0}^\infty\prod_{j=0}^{\infty}A
$$
shows that the system $(t^i\ctensor t^j)_{i\geq0,j\geq 0}$ is a topological basis of $B\ctensor_A B$ as pseudocompact $A$-module. The same argument also works for $B^ {\sigma}\ctensor_A B$. Let $a_{i,j}\in A$ with $a_{i,j}=0$ for $i<0$ or $j<0$. An easy inductive argument shows that
$$
\kappa\left(\sum_{i=0}^{\infty}\sum_{j=0}^{\infty} a_{i,j}t^i\ctensor t^j\right)=\sum_{i=0}^{\infty}\sum_{j=0}^{\infty} (a_{i-1,j}-a_{i,j-1})t^i\ctensor t^j
$$
is zero if and only if $a_{i,j}=0$ for all $i$ and $j$. Since the completed tensor product commutes with arbitrary products, it follows that  $\kappa$ is injective for all topologically free pseudocompact $B$-modules $M$.

If $M$ is an arbitrary pseudocompact $B$-module, Prop.~\ref{prop:free modules} implies the existence of an exact sequence
$$
0\mto N\mto F\mto M\mto 0
$$
of pseudocompact $B$-modules with $F$ topologically free. We obtain the following commutative diagram with exact rows and columns:
$$
\xymatrix{
&                            &                    & 0\ar[d] \\
&B^{\sigma}\ctensor_A N\ar[r]\ar[d]& B\ctensor_A N\ar[r]\ar[d]& N\ar[r]\ar[d]& 0\\
0\ar[r]&B^{\sigma}\ctensor_A F\ar[r]\ar[d]&B\ctensor_A F\ar[r]\ar[d]& F\ar[r]\ar[d] & 0\\
&B^{\sigma}\ctensor_A M\ar[r]^{\kappa}\ar[d]& B\ctensor_A M\ar[r]\ar[d]& M\ar[r]\ar[d] & 0\\
& 0& 0 & 0}
$$
By the Snake lemma, we see that $\kappa$ is injective in general.
\end{proof}

\subsection{K-theory}

We recall that the Quillen $\KTh$-groups of a ring $R$ may be calculated via Waldhausen's $S$-construction from the Waldhausen categories of either perfect or strictly perfect complexes \cite{Wal:AlgKTheo}, \cite{ThTr:HAKTS+DC}.

\begin{defn}
A complex of $R$-modules $\cmplx{P}$ is \emph{strictly perfect} if each $P^n$ is projective and finitely generated and if $P^n=0$ for almost all $n$. We denote by $\cat{SP}(R)$ the Waldhausen category of strictly perfect complexes with quasi-isomorphisms as weak equivalences and injections with strictly perfect cokernel as cofibrations.
\end{defn}

\begin{defn}
A complex of $R$-modules $\cmplx{P}$ is \emph{perfect} if it is quasi-isomorphic to a strictly perfect complex. We denote by $\cat{P}(R)$ the Waldhausen category of perfect complexes. Quasi-isomorphisms are weak equivalences and injections are cofibrations.
\end{defn}

Let $A$ be a left pseudocompact ring and $\cmplx{P}$ a strictly perfect complex. By Prop.~\ref{prop:finitely presented modules},
each $P^n$ carries a natural pseudocompact topology and thus becomes a projective object in $\cat{PC}(A)$. Moreover, if $\cmplx{Q}$ is a complex of pseudocompact $A$-modules (with continuous differentials) and $a\colon\cmplx{P}\mto\cmplx{Q}$ is a morphism of complexes, then $a$ is automatically continuous.

\begin{defn}
Let $A$ be a left pseudocompact ring. We call a complex $\cmplx{P}$ in the category of complexes over $\cat{PC}(A)$ \emph{topologically strictly perfect} if $\cmplx{P}$ is quasi-isomorphic to a strictly perfect complex, each $P^n$ is projective in $\cat{PC}(A)$, and $P^n=0$ for almost all $n$. We denote by $\cat{TSP}(A)$ the Waldhausen category of topologically strictly perfect complexes with quasi-isomorphisms as weak equivalences and injections with topologically strictly perfect cokernel as cofibrations.
\end{defn}

By what was said above and the Waldhausen approximation theorem \cite[Thm.~1.9.1]{ThTr:HAKTS+DC} the natural functors
$$
\cat{SP}(A)\mto \cat{TSP}(A),\qquad \cat{TSP}(A)\mto \cat{P}(A)
$$
induce isomorphisms on the associated $\KTh$-groups. The Waldhausen $\KTh$-groups of $\cat{SP}(A)$
are known to agree with the Quillen $\KTh$-groups of the ring $A$, see \cite[Thm.~1.11.2, Thm.~1.11.7.]{ThTr:HAKTS+DC}. Hence,
$$
\KTh_n(A)=\KTh_n(\cat{SP}(A))=\KTh_n(\cat{TSP}(A))=\KTh_n(\cat{P}(A))
$$
for all $n$. For us, it is most convenient to work with the category $\cat{TSP}(A)$.

\begin{lem}\label{lem:completed base change for perfect complexes}
Let $B$ be a left pseudocompact ring, $A$ be a pseudocompact ring and $M$ be a pseudocompact $B$-$A$-bimodule which is finitely generated projective as $B$-module.
For any $\cmplx{P}$ in $\cat{TSP}(A)$, the complex $M\ctensor_A \cmplx{P}$ is in $\cat{TSP}(B)$ and the functor
$$
\cat{TSP}(A)\mto\cat{TSP}(B),\qquad \cmplx{P}\mapsto M\ctensor_A\cmplx{P}
$$
is Waldhausen exact.
\end{lem}
\begin{proof}
Since the projectives in $\cat{PC}(B)$ are the direct summands of products of copies of $B$ (Prop.~\ref{prop:free modules}), $M\ctensor_A A=M$ (Prop.~\ref{prop:comparison of tensor prods}.(1)), and the completed tensor product commutes with products (Prop.~\ref{prop:commuting with filtered limits}) we conclude that $M\ctensor_A P^n$ is a projective pseudocompact $B$-module for each $n$.
Let $a\colon\cmplx{Q}\mto \cmplx{P}$ be quasi-isomorphism with $\cmplx{Q}$ a strictly perfect complex of $A$-modules. Since we are in the situation of Prop.~\ref{prop:tensor right exact},
we conclude that
$$
M\ctensor_A\cmplx{Q}\xrightarrow{\id\tensor a} M\ctensor_A\cmplx{P}
$$
is a quasi-isomorphism. On other hand, Prop.~\ref{prop:comparison of tensor prods} implies that
$$
M\tensor_A\cmplx{Q}\mto M\ctensor_A\cmplx{Q}
$$
is an isomorphism. Since $M$ is finitely generated and projective as $B$-module, the complex $M\tensor_A\cmplx{Q}$ is strictly perfect as complex of $B$-modules. This shows that $M\ctensor_A\cmplx{P}$ is an object of $\cat{TSP}(B)$.

By Prop.~\ref{prop:tensor right exact}, the completed tensor product with $M$ is a right exact functor on $\cat{PC}(A)$. Since the objects in $\cat{TSP}(A)$ are bounded complexes of projective objects in $\cat{PC}(A)$, it takes short exact sequences of complexes in $\cat{TSP}(A)$ to short exact sequences in $\cat{TSP}(B)$ and acyclic complexes to acyclic complexes. Therefore, it constitutes a Waldhausen exact functor between the categories (see e.\,g.\ \cite[Lemma~3.1.8]{Witte:PhD}).
\end{proof}

\begin{defn}
Let $A$ be a subring of $B$. We let $\cat{SP}^A(B)$ denote the Waldhausen subcategory of $\cat{SP}(B)$ consisting of strictly perfect complexes of $B$-modules which are perfect as complexes of $A$-modules. Furthermore, we let $\cat{SP}_A(B)$ denote the Waldhausen category with the same objects and cofibrations as $\cat{SP}(B)$, but with new weak equivalences: A morphism of complexes $f\colon \cmplx{P}\mto\cmplx{Q}$ is a weak equivalence if the cone of $f$ is in $\cat{SP}^A(B)$. If additionally $A$ and $B$ are left pseudocompact rings such that $B$ is pseudocompact as left $A$-module, we define likewise the Waldhausen categories $\cat{TSP}^A(B)$ and $\cat{TSP}_A(B)$.
\end{defn}

We recall from \cite[Thm.~1.8.2]{ThTr:HAKTS+DC} that the natural Waldhausen functors
$$
\cat{SP}^A(B)\mto\cat{SP}(B)\mto\cat{SP}_A(B)
$$
induce via the $S$-construction a homotopy fibre sequence of the associated topo\-logi\-cal spaces and hence, a long exact sequence
$$
\mto\KTh_n(\cat{SP}^A(B))\mto \KTh_n(B)\mto \KTh_n(\cat{SP}_A(B))\mto \dotsc \mto \KTh_0(B)\mto \KTh_0(\cat{SP}_A(B))\mto 0.
$$
The same is also true for the $\cat{TSP}$-version and by the approximation theorem \cite[Thm.~1.9.1]{ThTr:HAKTS+DC} the resulting long exact sequence is isomorphic to the one above.

\begin{thm}\label{thm:splitting abstract}
Let $B$ be a skew power series ring over a pseudocompact ring $A$ and assume that there exists a unit $\gamma\in B$ such that
$$
\sigma(a)=\gamma a \gamma^{-1}
$$
for all $a\in A$. The long exact localisation sequence for $\cat{TSP}^A(B)\mto\cat{TSP}(B)\mto\cat{TSP}_A(B)$ splits into short split-exact sequences
$$
0\mto \KTh_{n+1}(B)\mto \KTh_{n+1}(\cat{TSP}_A(B))\mto \KTh_n(\cat{TSP}^A(B))\mto 0
$$
for $n\geq 0$.
\end{thm}
\begin{proof}
For any Waldhausen category $\cat{W}$, let $\Kspace(\cat{W})$ denote the topological space associated to $\cat{W}$ via the $S$-construction \cite[Def.~1.5.3]{ThTr:HAKTS+DC}. Consider the homotopy fibre sequence
$$
\Kspace(\cat{TSP}^A(B))\mto\Kspace(\cat{TSP}(B))\mto\Kspace(\cat{TSP}_A(B)).
$$
from \cite[Thm.~1.8.2]{ThTr:HAKTS+DC}.

Note that right multiplication with $\gamma^{-1}$ gives rise to an isomorphism of $B$-$A$-bimodules $B\mto B^\sigma$.
By Lemma~\ref{lem:completed base change for perfect complexes} and since every projective pseudocompact $B$-module is also projective as pseudocompact $A$-module the completed tensor product
$$
\cat{TSP}^A(B)\mto\cat{TSP}(B),\qquad \cmplx{P}\mapsto B\ctensor_A\cmplx{P}
$$
is a Waldhausen exact functor.
By Prop.~\ref{prop:the exact sequence}
\begin{equation}\label{eq:exact sequence}
0\mto B\ctensor_A \cmplx{P}\xrightarrow{\kappa\comp\gamma^{-1}} B\ctensor_A\cmplx{P}\mto \cmplx{P}\mto 0
\end{equation}
is an exact sequence in $\cat{TSP}(B)$. The additivity theorem \cite[Cor.~1.7.3]{ThTr:HAKTS+DC} applied to this sequence implies that the map $\Kspace(\cat{TSP}^A(B))\mto\Kspace(\cat{TSP}(B))$ induced by the canonical inclusion
is homotopic to the constant map defined by the base point. By \cite[Thm.~6.8]{Bredon:TopologyAndGeometry}, which applies equally well to homotopy fibre sequences instead of fibre sequences, the assertion of the theorem follows.
\end{proof}

\begin{rem}
Let $B=A[[t;\sigma,\delta]]$ be a skew power series ring with a unit $\gamma\in B^{\times}$ as above. For $a\in A^{\times}$ and $b\in A$, let
$B'=A[[t';\sigma',\delta']]$ be the topological ring $B$ with the modified skew power series ring structure as in Remark~\ref{rem:change of variable} and set $\gamma'=a\gamma$. Then the exact sequence $(\ref{eq:exact sequence})$ for $B'$ is exactly the same as the one for $B$. In particular, one obtains the same splitting for the $\KTh$-groups.
\end{rem}

Using the algebraic description  of the first Postnikov section of the topological space associated to a Waldhausen category \cite[Def. 1.2]{MT:1TWKTS} we can give the following explicit formula for a section $s\colon\KTh_0(\cat{TSP}^A(B))\mto\KTh_1(\cat{TSP}_A(B))$: The group $\KTh_0(\cat{TSP}^A(B))$ is the abelian group generated by the objects of $\cat{TSP}^A(B)$ modulo the relations generated by exact sequences and weak equivalences in $\cat{TSP}^A(B)$. Moreover, there exists a canonical class $[w]\in\KTh_1(\cat{TSP}_A(B))$ for each weak auto-equivalence $w\colon\cmplx{P}\mto\cmplx{P}$ in $\cat{TSP}_A(B)$.

\begin{defn}\label{defn:explicit splitting}
In the situation of Theorem~\ref{thm:splitting abstract}, set
$$
s_{B,\gamma}(\cmplx{P})=[B\ctensor_A \cmplx{P}\xrightarrow{\kappa\comp\gamma^{-1}} B\ctensor_A\cmplx{P}]^{-1}\in\KTh_1(\cat{TSP}_A(B))
$$
for each $\cmplx{P}$ in $\cat{TSP}^A(B)$.
\end{defn}

It is easy to check that $s_{B,\gamma}$ defines a group homomorphism
$$
\KTh_0(\cat{TSP}^A(B))\mto\KTh_1(\cat{TSP}_A(B))
$$
(see e.\,g.\ \cite[Thm~2.4.7]{Witte:PhD}). Moreover, $\del_0(s_{B,\gamma}(\cmplx{P}))$ is the cone of $\kappa\comp\gamma^{-1}$ by \cite[Thm.~A.5]{Witte:MCVarFF} and hence, equal to $\cmplx{P}$ in  $\KTh_0(\cat{TSP}^A(B))$.

\begin{rem}\label{rem:normalisation of boundary}
Note that the above description of the boundary homomorphism $\del_0$ is compatible with the formula given in \cite[p.~2]{WeibYao:Localization}. Other authors prefer to use $-\del_0$ instead.
\end{rem}

We record the following functorial behaviour.

\begin{prop}\label{prop:functoriality of splitting}
Let $B_i=A_i[[t_i;\sigma_i,\delta_i]]$ ($i=1,2$) be two skew power series rings, $M$ a pseudocompact $B_2$-$B_1$-bimodule which is finitely generated and projective as $B_2$-module, $N\subset M$ a pseudocompact $A_2$-$A_1$-subbimodule which is finitely generated and projective as $A_2$-module such that  $N\ctensor_{A_1} B_1\isomorph M$ as $A_2$-$B_1$-bimodules. Then the completed tensor product with $M$ induces Waldhausen exact functors
\begin{align*}
T_M\colon\cat{TSP}^{A_1}(B_1)&\mto\cat{TSP}^{A_2}(B_2),\qquad \cmplx{P}\mapsto M\ctensor_{B_1}\cmplx{P},\\
T_M\colon\cat{TSP}_{A_1}(B_1)&\mto\cat{TSP}_{A_2}(B_2),\qquad \cmplx{P}\mapsto M\ctensor_{B_1}\cmplx{P}.
\end{align*}
Assume further that there exists units $\gamma_i\in B_i$ with $\sigma_i(a)=\gamma_i a \gamma_i^{-1}$ for all $a\in A_i$ and let one of the following two conditions be satisfied:
\begin{enumerate}
\item The inclusion $N\subset M$ also induces an isomorphism of pseudocompact $B_2$-$A_1$-bimodules $B_2\ctensor_{A_2}N\isomorph M$ and for all $n\in N$,
    $$
    \gamma_2 n\gamma_1^{-1}\in N ,\qquad t_2 n-\gamma_2 n \gamma_1^{-1}t_1\in N.
    $$
\item  $A_2=A_1$ and $B_2\subset B_1$ is a closed subring; $t_i=\gamma_i-1$ for $i=1,2$; there exists a number $k\in\Int$ such that $\gamma_2=\gamma_1^k$; $M=B_1$ and $N=A_1$.
\end{enumerate}
Then the following diagram is commutative:
$$
\xymatrix{
\KTh_0(\cat{TSP}^{A_1}(B_1))\ar[r]^{s_{B_1,\gamma_1}}\ar[d]^{T_M}&\KTh_1(\cat{TSP}_{A_1}(B_1))\ar[d]^{T_M}\\
\KTh_0(\cat{TSP}^{A_2}(B_2))\ar[r]^{s_{B_2,\gamma_2}}&\KTh_2(\cat{TSP}_{A_2}(B_2))
}
$$
\end{prop}
\begin{proof}
Let $\cmplx{P}$ be in $\cat{TSP}^{A_1}(B_1)$. By definition, there exists a strictly perfect complex of $A_1$-modules $\cmplx{Q}$ and a quasi-isomorphism $\cmplx{Q}\mto\cmplx{P}$ in $\cat{TSP}(A_1)$. Hence,
$$
N\tensor_{A_1}\cmplx{Q}\mto N\ctensor_{A_1}\cmplx{P}\isomorph M\ctensor_{B_1}\cmplx{P}
$$
is a quasi-isomorphism of complexes of pseudocompact $A_2$-modules. This shows that $T_M$ takes objects of $\cat{TSP}^{A_1}(B_1)$ to objects of $\cat{TSP}^{A_2}(B_2)$ and weak equivalences in $\cat{TSP}_{A_1}(B_1)$ to weak equivalences in $\cat{TSP}_{A_2}(B_2)$. It follows easily that both versions of $T_M$ are Waldhausen exact as claimed.

Assume that condition $(1)$ is satisfied. For $n\in N$ set $\sigma(n)=\gamma_2 n\gamma_1^{-1}$, $\delta(n)=t_2n-\sigma(n)t_1$ and consider the homomorphism of $B_2$-$B_1$-bimodules
\begin{gather*}
\kappa'\colon B_2\ctensor_{A_2}N\ctensor_{A_1}B_1\mto B_2\ctensor_{A_2}N\ctensor_{A_1}B_1,\\
b_2\ctensor n\ctensor b_1\mapsto b_2\gamma_2^{-1}t_2\ctensor n\ctensor b_1-b_2\gamma_2^{-1}\ctensor\sigma(n)\ctensor t_1 b_1-b_2\gamma_2^{-1}\ctensor\delta(n)\ctensor b_1.
\end{gather*}
Then the diagram
$$
\xymatrix{
M\ctensor_{B_1}B_1\ctensor_{A_1}\cmplx{P}\ar[r]^{\id_M\ctensor(\kappa\comp\gamma_1^{-1})}\ar@{=}[d]&M\ctensor_{B_1}B_1\ctensor_{A_1}\cmplx{P}\ar@{=}[d]\\
M\ctensor_{A_1}\cmplx{P}\ar[r]&M\ctensor_{A_1}\cmplx{P}\\
B_2\ctensor_{A_2}N\ctensor_{A_1}B_1\ctensor_{B_1}\cmplx{P}\ar[u]_{\isomorph}\ar[d]^{\isomorph}\ar[r]^{\kappa'\ctensor\id_{\cmplx{P}}}&
B_2\ctensor_{A_2}N\ctensor_{A_1}B_1\ctensor_{B_1}\cmplx{P}\ar[u]_{\isomorph}\ar[d]^{\isomorph}\\
B_2\ctensor_{A_2}M\ctensor_{B_1}\cmplx{P}\ar[r]^{(\kappa\comp\gamma_2^{-1})\ctensor\id_{\cmplx{P}}}&B_2\ctensor_{A_2}M\ctensor_{B_1}\cmplx{P}
}
$$
commutes for every $\cmplx{P}$ in $\cat{TSP}^{A_1}(B_1)$. We conclude
\begin{align*}
T_M(s_{B_1,\gamma_1}(\cmplx{P}))&=[\id_M\ctensor(\kappa\comp\gamma_1^{-1})]^{-1}=[\kappa'\ctensor\id_{\cmplx{P}}]^{-1}\\
&=[(\kappa\comp\gamma_2^{-1})\ctensor\id_{\cmplx{P}}]^{-1}=s_{B_2,\gamma_2}(T_M(\cmplx{P}))
\end{align*}
from the relations listed in \cite[Def.~1.2]{MT:1TWKTS}.

The proof in the case that condition $(2)$ is satisfied follows essentially as in \cite[Proof of Lemma~4.6]{BV:DescentTheory}:
Let $A=A_1=A_2$ and
set for $\cmplx{P}$ in $\cat{TSP}^{A}(B_1)$
\begin{align*}
\Delta_i\colon B_2\ctensor_{A}\cmplx{P}&\mto B_1\ctensor_{A}\cmplx{P},\qquad x\ctensor y\mapsto x\gamma_1^{-i}\ctensor \gamma_1^{i}y,\\
\kappa_1\colon B_1\ctensor_{A}\cmplx{P}&\mto B_1\ctensor_{A}\cmplx{P},\qquad x\ctensor y\mapsto x\gamma_1^{-1}t_1\ctensor y -x\gamma_1^{-1}\ctensor t_1=x\ctensor y-x\gamma_1^{-1}\ctensor \gamma_1 y,\\
\kappa_2\colon B_2\ctensor_{A}\cmplx{P}&\mto B_1\ctensor_{A}\cmplx{P},\qquad x\ctensor y\mapsto x\gamma_2^{-1}t_2\ctensor y -x\gamma_2^{-1}\ctensor t_2=x\ctensor y-x\gamma_2^{-1}\ctensor \gamma_2 y.
\end{align*}
Then
$$
(B_2\ctensor_{A}\cmplx{P})^k\mto B_1\ctensor_{A}\cmplx{P},\qquad (x_i)_{i=0}^{k-1}\mapsto \sum_{i=0}^{k-1}\Delta_i(x_i),
$$
is an isomorphism in $\cat{TSP}(B_2)$. Using this as an identification, $\kappa_1$ is given by right multiplication with the $k\times k$-matrix of $B_2\ctensor_{A}\cmplx{P}$-endomorphisms
\begin{gather*}
\begin{pmatrix}
\id       & -\id  & 0     &\hdots  &0\\
0         & \ddots&\ddots & \ddots & \vdots\\
\vdots    & \ddots&\ddots  & \ddots & 0 \\
0         & \ddots& \ddots &\ddots  &-\id\\
-\Delta_k & 0     &\hdots & 0      &\id
\end{pmatrix}=\\
\begin{pmatrix}
\id       & 0  & \hdotsfor{2}  &0\\
0         & \id&  0 &  \hdots & 0\\
\vdots    & \ddots&  \ddots  & \ddots & \vdots \\
0         & \hdots&  0     &\id  &   0\\
-\Delta_k & \hdotsfor{2}   & -\Delta_k &\id-\Delta_k\\
\end{pmatrix}
\begin{pmatrix}
\id       & -\id  & 0     &\hdots  &0\\
0         & \id   & -\id  & \ddots & \vdots\\
\vdots    & 0     & \ddots& \ddots & 0 \\
\vdots    & \vdots& \ddots&\ddots  &-\id\\
0         & 0     &\hdots & 0      &\id
\end{pmatrix}.
\end{gather*}
Note that $\id-\Delta_k=\kappa_2$.
The relations listed in \cite[Def.~1.2]{MT:1TWKTS} imply that the class in $\KTh_1(\cat{TSP}_A(B_2))$ of a triangular matrix as above is the product of the classes of the diagonal entries and the class of a composition of weak auto-equivalences is the product of the classes of the weak auto-equivalences. Hence,
$$
T_M(s_{B_1,\gamma_1}(\cmplx{P}))=[\kappa_1]^{-1}=[\id-\Delta_k]^{-1}=s_{B_2,\gamma_2}(T_M(\cmplx{P})).
$$
\end{proof}

\subsection{Left Denominator Sets}

We recall that a subset $S$ in a ring $B$ is a \emph{left denominator set} if
\begin{enumerate}
\item $S$ is multiplicatively closed,
\item (\emph{Ore condition}) for each $s\in S$, $b\in B$ there exist $s'\in S$, $b'\in B$ such that $b's=s'b$, and
\item (\emph{annihilator condition}) for each $s\in S$, $b\in B$ with $bs=0$ there exists $s'\in S$ with $s'b=0$.
\end{enumerate}
If $S$ is a left denominator set, the \emph{(left) quotient ring} $B_S$ of $B$ exists and is flat as right $B$-module \cite[Ch.~10]{GW:NoncommNoethRings}.

We would like to identify the $\KTh$-groups of the Waldhausen category $\cat{SP}_A(B)$ with the $\KTh$-groups of the quotient ring $B_S$ of $B$ with respect to a left denominator set $S\subset B$. Unfortunately, it is not clear to us that this is always possible. Below, we will show that there exists an essentially unique candidate for $S$. Under the condition that $A$ is noetherian plus a mild extra condition we show that this $S$ does indeed have the right properties.

\begin{defn}\label{defn:can Ore set}
Let $A$ be a subring of $B$. We let $S$ denote the set of elements $b\in B$ such that
$$
B\xrightarrow{\cdot b} B
$$
is perfect as complex of $A$-modules and hence, an object in $\cat{SP}^A(B)$.
\end{defn}

\begin{lem}
The set $S$ is multiplicatively closed and saturated, i.\,e.\ if any two of the elements $a, b, ab\in B$ are in $S$, then so is the third.
\end{lem}
\begin{proof}
The commutative diagram
$$
\xymatrix{
B\ar[r]^{=}\ar[d]^{\cdot a}&B\ar[r]^{\cdot a}\ar[d]^{\cdot ab}&B\ar[d]^{\cdot b}\\
B\ar[r]^{\cdot b}&B\ar[r]^{=}&B
}
$$
constitutes a distinguished triangle if one reads the columns as complexes. If in a distinguished triangle two of the three complexes are perfect, then so is the third.
\end{proof}

\begin{lem}\label{lem:uniqueness of S}
Assume that there exists a left denominator set $S'$ such that $\cat{SP}^A(B)$ is precisely the subcategory of those complexes in $\cat{SP}(B)$ whose cohomology modules are $S'$-torsion. Then the set $S$ has the same property and is the saturation of $S'$, i.\,e.\ it is the left denominator set consisting of those elements in $B$ that become a unit in $B_{S'}$.
\end{lem}
\begin{proof}
An element $b\in B$ is a unit in $B_{S'}$ if and only if the complex $B\xrightarrow{\cdot b}B$ is $S'$-torsion and thus perfect. Hence, $B_{S}=B_{S'}$ and a complex is $S$-torsion precisely if it is $S'$-torsion.
\end{proof}

The following theorem extends \cite[Thm. 2.25]{SchnVen:LocCompSPSR}.

\begin{thm}\label{thm:existence of Ore set}
Assume that $B$ is a skew power series ring over a noetherian pseudocompact coefficient ring $A$ and $I$ is an open two-sided ideal of $A$ stable under $\sigma$ and $\delta$ and contained in the Jacobson radical of $A$. Then $S$ is a left denominator set and that $\cat{SP}^A(B)$ is precisely the subcategory of those complexes in $\cat{SP}(B)$ whose cohomology modules are $S$-torsion.
\end{thm}
\begin{proof}
Since $A$ is noetherian, the ideal powers $I^k$ constitute a fundamental system of open neighbourhoods of 0 \cite[Sect.~VII]{Warner:TopRings}. Since $I$ is stable under $\sigma$ and $\delta$, we deduce that
$$
J_k=\prod_{n=0}^{\infty}I^kt^n
$$
is a closed two-sided ideal of $B$ for each positive integer $k$.
Consider the set
$$
S'=\bigcup_{n=0}^{\infty}t^n+J_1.
$$
Clearly, $S'$ is multiplicatively closed. Furthermore, the annihilator $\Ann_B t^n+f$ is zero for $f\in J_1$. Indeed, if $gt^n+gf=0$ then reduction modulo $J_k$ shows that the coefficients of $g$ must lie in
$$
\bigcap_{k=1}^\infty I^k=0.
$$
For any $f\in J$ we have
$$(B/B(t^n+f))/I(B/B(t^n+f))=B/Bt^n+J_1$$
which is clearly finitely generated as $A/I$-module. By the topological Nakayama lemma we conclude that $B/B(t^n+f)$ is finitely generated as $A$-module.

Assume that $M$ is a $B$-module which is finitely generated over $A$ and let $m\in M$. Since $A$ is noetherian, the module $Bm$ is also finitely generated over $A$. In particular, it is pseudocompact for the $I$-adic topology. Since $t$ is topologically nilpotent, there exists a positive integer $n$ such that $t^nm\in J_1m$, i.\,e.\ there exists a $g\in J_1$ such that $(t^n+g)m=0$.
Applying this to $B/B(t^n+f)$ we conclude that $S'$ satisfies the Ore condition, whereas the annihilator condition is automatically satisfied. Hence, $S'$ is a left denominator set.

Let now $\cmplx{P}$ be a strictly perfect complex of $B$-modules. Since $B$ is a direct product of free $A$-modules and $A$ is noetherian, any finitely generated projective $B$-module is flat as $A$-module. Hence, $\cmplx{P}$ has finite flat dimension over $A$ and is perfect as complex of $A$-modules if and only if its cohomology groups are finitely generated as $A$-modules. It follows immediately that every complex in $\cat{SP}^A(B)$ is $S'$-torsion.

Conversely, assume that $\cmplx{P}$ is $S'$-torsion. Without loss of generality we may assume that $\cmplx{P}$ is concentrated in the degrees $0$ to $n$, that $\HF^n(P)\neq 0$ and that all $P^i$ for $i>0$ are in fact free $B$-modules. We do induction on $n$. In the case that the finitely generated projective $B$-module $P^0$ happens to be $S'$-torsion we choose a system of generators of $P^0$ and an $s\in S'$ that kills all the generators and obtain a surjection $B^k/B^ks\mto P^0$, which shows that $P^0$ is also finitely generated and projective as $A$-module. If $n=1$ and $\HF^0(P)=0$ we proceed similarly with $\HF^1(P)$. If $\HF^0(P)\neq 0$ or if $n>1$ we find a morphism from a shift of the complex $B^k\xrightarrow{\cdot s}B^k$ to $\cmplx{P}$ such that we get a surjection onto $\HF^n(P)$. The cone of this morphism is strictly perfect as complex of $B$-modules, $S'$-torsion and concentrated in the degrees $0$ to $n-1$ if $n>1$. If $n=1$, it is concentrated in degrees $-1$ to $0$, but the zeroth cohomology group is the only one that might not vanish. This shows that $\cmplx{P}$ is in $\cat{SP}^A(B)$. We now apply Lemma~\ref{lem:uniqueness of S} to $S'$.
\end{proof}

\begin{rem}
If one drops the assumption that $A$ is noetherian one can still show that the annihilators of elements in the set $S'$ are trivial. For this, one uses that $I$ is transfinitely nilpotent \cite[Thm.~33.21]{Warner:TopRings}. Assuming that $S'$ also satisfies the Ore condition, the above proof shows that every $S'$-torsion complex is in $\cat{SP}^A(B)$. However, we were not able to prove the converse nor verify the Ore condition in this generality.
\end{rem}


\begin{prop}\label{prop:WeibYao}
Assume that $S$ is a left denominator set such that $\cat{SP}^A(B)$ is precisely the subcategory of those complexes in $\cat{SP}(B)$ whose cohomology modules are $S$-torsion. The Waldhausen exact functor
$$
\cat{SP}_A(B)\mto \cat{SP}(B_S),\qquad \cmplx{P}\mapsto B_S\tensor_B \cmplx{P},
$$
induces isomorphisms
$$
\KTh_n(\cat{SP}_A(B))\isomorph \KTh_n(B_S)
$$
for $n>0$. The groups $\KTh_n(\cat{SP}^A(B))$ can be identified with the relative $\KTh$-groups $\KTh_n(B,B_S)$ for all $n\geq 0$. \end{prop}
\begin{proof}
This is a direct consequence of the localisation theorem in \cite{WeibYao:Localization}.
\end{proof}

\begin{cor}\label{cor:splitting sps-rings with S}
Let $B$ be a skew power series ring over a noetherian pseudocompact ring $A$ and let $\gamma\in B$ be a unit such that
$$
\sigma(a)=\gamma a\gamma^{-1}
$$
for every $a\in A$. Assume further that $I$ is an open two-sided ideal of $A$ which is contained in the Jacobson radical and stable under $\sigma$ and $\delta$. Then $S$ is a left denominator set in $B$ and for every $n\geq 0$
$$
\KTh_{n+1}(B_S)\isomorph \KTh_{n+1}(B)\oplus\KTh_{n}(B,B_S)
$$
\end{cor}
\begin{proof}
 This follows from Theorem~\ref{thm:splitting abstract}, Theorem~\ref{thm:existence of Ore set} and Prop.~\ref{prop:WeibYao}.
\end{proof}

\begin{rem}\label{rem:existence of stable I}
If in the above corollary, $B$ is itself a (left and right) pseudocompact ring, then the existence of $I$ is automatic. Indeed, the definitions of $\sigma$ and $\delta$ can be extended to $B$ and every two-sided ideal of $B$ is then stable under $\sigma$ and $\delta$. The intersections of the open two-sided ideals of $B$ with $A$ form a fundamental system of neighbourhoods of $0$ in $A$. Since $A$ is noetherian, the Jacobson radical of $A$ is open. Hence, there exists an open two-sided ideal $J$ of $B$ such that $J\cap A$ is contained  in the Jacobson radical.
\end{rem}

\section{Applications to Completed Group Rings}\label{sec:Applications}

\begin{defn}
Let $G$ be a pro-finite group and $\mathcal{U}_G$ the directed set of open normal subgroups. For any topological ring $R$ we define the \emph{completed group ring} of $G$ over $R$ to be the topological ring
$$
R[[G]]=\varprojlim_{U\in\mathcal{U}_G}R[G/U].
$$
\end{defn}

We will now fix a pseudocompact coefficient ring $R$ and a prime $p$ which is topologically nilpotent in $R$. Furthermore, let $G$ be a pro-finite group which is the semi-direct product of a closed normal subgroup $H$ and a closed subgroup $\Gamma\isomorph\Int_p$. We also fix a topological generator $\gamma\in \Gamma$. The following proposition is a harmless generalisation of \cite[Ex. 5.1]{Ven:CharElements}.

\begin{prop}
The completed group ring $R[[G]]$ is a left and right pseudocompact skew power series ring over $R[[H]]$ with
\begin{align*}
t&=\gamma-1,\\
\sigma(a)&=\gamma a\gamma^{-1},\\
\delta(a)&=\sigma(a)-a
\end{align*}
for $a\in R[[H]]$.
\end{prop}
\begin{proof}
It is obvious that the completed group rings are pseudocompact and that $R[[G]]$ is pseudocompact as $R[[H]]$-$R[[H]]$-bimodule. Since $H$ is normal, it is also easy to verify that $\sigma$ is a continuous ring automorphism of $R[[H]]$ and that $\delta$ is a continuous $\sigma$-derivation.
It remains to show that
$$
R[[G]]=\prod_{n=0}^{\infty}R[[H]]t^n
$$
as $R[[H]]$-module. The topological nilpotence of $p$ implies that $R[[H]]$ is pro-discrete over $\Int_p$. One then verifies that
$$
R[[H]]\ctensor_{\Int_p}\Int_p[[\Gamma]]\mto R[[G]],\qquad a\ctensor x\mapsto ax,
$$
is an isomorphism of pseudocompact $R[[H]]$-modules. By a classical result from Iwasawa theory \cite[Thm.~7.1]{Wa:ICF}
$$
\Int_p[[\Gamma]]\mto \Int_p[[t]],\qquad \gamma\mapsto t+1,
$$
defines a topological ring isomorphism. Hence, the system $(t^n)_{n\geq 0}$ is a topological $R[[H]]$-basis of $R[[G]]$.
\end{proof}

\begin{cor}
Let $R$, $G$, and $H$ be as above. For any $n\geq 0$,
$$
\KTh_{n+1}(\cat{TSP}_{R[[H]]}(R[[G]]))\isomorph \KTh_{n+1}(R[[G]])\oplus \KTh_{n}(\cat{TSP}^{R[[H]]}(R[[G]])).
$$
\end{cor}
\begin{proof}
See Theorem~\ref{thm:splitting abstract}.
\end{proof}

\begin{cor}
Suppose that $R$ is a commutative noetherian pseudocompact ring with $p\in R$ topologically nilpotent and that $G$ is a compact $p$-adic Lie  group. Then $R[[H]]$ and $R[[G]]$ are noetherian, the set $S$ of Definition~\ref{defn:can Ore set} is a left denominator set, and for any $n\geq 0$,
$$
\KTh_{n+1}(R[[G]]_S)\isomorph\KTh_{n+1}(R[[G]])\oplus \KTh_{n}(R[[G]],R[[G]]_S).
$$
\end{cor}
\begin{proof}
Every commutative noetherian pseudocompact ring is a finite direct product of commutative noetherian pseudocompact local rings with the topology given by the powers of the maximal ideal \cite[Cor.~ 36.35]{Warner:TopRings}. Hence, we may suppose that $R$ is local. As a corollary of Cohen's structure theorem, we have
$$
R\isomorph\mathcal{O}[[X_1,\dots,X_n]]/I
$$
for a complete discrete valuation ring $\mathcal{O}$ with maximal ideal $p\mathcal{O}$, some integer $n\geq 0$, and some ideal $I$ (which is finitely generated and therefore closed) \cite[Ch.~IX, \S~2, Prop.~1, Thm.~3]{Bourbaki:CommAlg}. Hence, $R[[G]]$ is a factor ring of
$$
\mathcal{O}[[X_1,\dots,X_n]][[G]]\isomorph\mathcal{O}[[G\times\prod_{k=1}^n\Gamma]],\qquad \Gamma\isomorph\Int_p.
$$
So, we may assume that $R=\mathcal{O}$. By \cite[Thm.~27.1]{Schneider:LieGroups} we can find an open $p$-valuable subgroup $G'$ of $G$. By \cite[Thm.~33.4]{Schneider:LieGroups}, $\mathcal{O}[[G']]$ is noetherian. Since $\mathcal{O}[[G]]$ is finitely generated as left or right $\mathcal{O}[[G']]$-module, $\mathcal{O}[[G]]$ is also noetherian. Now apply Corollary~\ref{cor:splitting sps-rings with S} together with Remark~\ref{rem:existence of stable I}.
\end{proof}

If $R=\Int_p$, it is not hard to verify that our set $S$ coincides with Venjakob's canonical Ore set $S$ introduced in \cite{CFKSV}.

The following corollary gives a new and more conceptual proof of the central \cite[Thm.~6.1.(ii)]{Burns:MCinNCIwasawaTh+RelConj}. Moreover, we can dispose of the assumption that the coefficient ring $R$ is unramified over $\Int_p$. In particular, one can replace in \cite[Thm.~2.2,  Cor.~2.3]{Burns:MCinNCIwasawaTh+RelConj} the base ring $\Int_p$ by the valuation ring $\mathcal{O}$ of any finite extension field of $\Rat_p$.

Let $G=H\rtimes \Gamma$ be a $p$-adic Lie group as above. Write $\SQab(G)$ for the set of pairs $\tau=(U_\tau,J_\tau)$ with $U_\tau \subset G$ an open subgroup, $J_\tau\subset U_\tau\cap H$ a subgroup which is open in $H$ and normal in $U_\tau$ and such that $G_\tau=U_\tau/J_\tau$ is an abelian $p$-adic Lie group of rank $1$. Furthermore, let $d_\tau$ be the index of $HU_\tau$ in $G$ and $H_\tau$ be the image of $H\cap U_\tau$ in $G_\tau$. For each $\tau$, we also fix once and for all a $\gamma_{\tau}$ lying in a $p$-Sylow subgroup of $U_\tau$ such that $\gamma^{-d_\tau}\gamma_{\tau}\in H$. We write $\Gamma_\tau\isomorph \Int_p$ for the closed subgroup of $U_\tau$ topologically generated by $\gamma_\tau$ and identify it with its image in $G_\tau$. Note that $H_\tau$ is a commutative finite group and that the choice of $\gamma_\tau$ induces decompositions $U_\tau=U_\tau\cap H\rtimes \Gamma_{\tau}$ and $G_\tau=H_\tau\times \Gamma_\tau$.

From now on, we assume that $R$ is commutative, noetherian, and pseudocompact with $p\in R$ topologically nilpotent.
Since
$$
R[[G_\tau]]\ctensor_{R[[U_\tau]]}R[[G]]\isomorph R[H_\tau]\ctensor_{R[[J_\tau]]}R[[G]]
$$
as pseudocompact $R[H_\tau]$-$R[[G]]$-bimodules, the exact functor
$$
\cat{TSP}(R[[G]])\mto\cat{TSP}(R[[G_\tau]]),\qquad \cmplx{P}\mapsto R[[G_\tau]]\ctensor_{R[[U_\tau]]}\cmplx{P}
$$
induces exact functors
\begin{align*}
\cat{TSP}_{R[[H]]}(R[[G]])&\mto \cat{TSP}_{R[H_\tau]}(R[[G_\tau]]),\\
\cat{TSP}^{R[[H]]}(R[[G]])&\mto \cat{TSP}^{R[H_\tau]}(R[[G_\tau]])
\end{align*}
by the first part of Prop.~\ref{prop:functoriality of splitting}. We will denote by $q_\tau$ the induced homomorphisms of $\KTh$-groups.

For each $\phi$  in the group of characters $\hat{H}_\tau$ of $H_\tau$, let $R[\phi]$ denote the finite extension of $R$ generated by the values of $\phi$. Consider the ring homomorphism
$$
R[[G_\tau]]=R[[\Gamma_\tau]][H_\tau]\mto R[\phi][[\Gamma_\tau]]
$$ which maps $h\in H_\tau$ to $\phi(h)$. Thus, $R[\phi][[\Gamma_\tau]]$ becomes a pseudocompact $R[\phi][[\Gamma_\tau]]$-$R[[G_\tau]]$-bimodule and we obtain an exact functor
$$
\cat{TSP}(R[[G_\tau]])\mto \cat{TSP}(R[\phi][[\Gamma_\tau]]),\qquad \cmplx{P}\mapsto R[\phi][[\Gamma_\tau]]\ctensor_{R[[G_\tau]]}\cmplx{P}.
$$
Since
$
R[\phi][[\Gamma_\tau]]\isomorph R[\phi]\ctensor_{R[H_\tau]}R[[G_\tau]]
$
as $R[\phi]$-$R[[G_\tau]]$-bimodules we also get exact functors
\begin{align*}
\cat{TSP}_{R[H_\tau]}(R[[G_\tau]])&\mto \cat{TSP}_{R[\phi]}(R[\phi][[\Gamma_\tau]]),\\
\cat{TSP}^{R[H_\tau]}(R[[G_\tau]])&\mto \cat{TSP}^{R[\phi]}(R[\phi][[\Gamma_\tau]]).
\end{align*}
We will denote the induced homomorphisms of $\KTh$-groups by $q_\phi$.

Recall that $R[\phi][\Gamma_\tau$ and $R[\phi][[\Gamma_\tau]]_S$ are semilocal commutative rings. Hence, the determinant induces isomorphisms
\begin{align*}
\KTh_1(\cat{TSP}(R[\phi][[\Gamma_\tau]]))&=\KTh_1(R[\phi][[\Gamma_\tau]])\isomorph R[\phi][[\Gamma_\tau]]^{\times}, \\
\KTh_1(\cat{TSP}^{R[\phi]}(R[\phi][[\Gamma_\tau]]))&=
\KTh_1(R[\phi][[\Gamma_\tau]]_{S})\isomorph R[\phi][[\Gamma_\tau]]_{S}^{\times}.
\end{align*}
Consider the homomorphisms
\begin{align*}
\Delta\colon& \KTh_1(R [[G]])\mto\prod_{\tau\in\SQab(G)}\prod_{\phi\in\hat{H}_\tau} R[\phi][[\Gamma_\tau]]^{\times}\\
\Delta_S\colon&\KTh_1(R[[G]]_S)\mto\prod_{\tau\in\SQab(G)}\prod_{\phi\in\hat{H}_\tau} R[\phi][[\Gamma_\tau]]_{S}^{\times}\\
\end{align*}
whose $(\tau,\phi)$-component is $q_\phi\comp q_\tau$ and note that they agree with the maps $\Delta_{\mathcal{O},G}$ and $\Delta_{\mathcal{O},G,S}$ from \cite[Thm.~6.1]{Burns:MCinNCIwasawaTh+RelConj} in the case that $R=\mathcal{O}$ is the valuation ring of a finite extension of $\Rat_p$. (We also note that they depend on the splittings $G_\tau=H_\tau\times \Gamma_\tau$ induced by the choices of the generators $\gamma_\tau$.)

\begin{cor}
Let $G=H\rtimes \Gamma$ be a $p$-adic Lie group and $R$ be any commutative, noetherian, pseudocompact ring such that $p\in R$ is topologically nilpotent. Then
$$\im \Delta_S\cap\prod_{\tau\in\SQab(G)}\prod_{\phi\in\hat{H}_\tau} R[\phi][[\Gamma_\tau]]^{\times}=\im \Delta.$$
\end{cor}
\begin{proof}
We use the explicit description of the splitting
$$
s=s_{G,\gamma}\colon\KTh_0(\cat{TSP}^{R[[H]]}(R[[G]]))
  \mto\KTh_1(\cat{TSP}_{R[[H]]}(R[[G]]))
$$
given in Definition~\ref{defn:explicit splitting}. It suffices to prove that the diagram
$$
\xymatrix{
\KTh_0(\cat{TSP}^{R[[H]]}(R[[G]]))\ar[r]^{s_{G,\gamma}}\ar[d]^{(q_\phi\comp q_\tau)}&
\KTh_1(\cat{TSP}_{R[[H]]}(R[[G]]))\ar[d]^{\Delta_S}\\
\prod_{\tau,\phi}\KTh_0(\cat{TSP}^{R[\phi]}(R[\phi][[\Gamma_\tau]]))\ar[r]^{(s_{\Gamma_\tau,\gamma_\tau})}&
\prod_{\tau,\phi}\KTh_1(\cat{TSP}_{R[\phi]}(R[\phi][[\Gamma_\tau]]))}
$$
commutes.

For each integer $d$, we let $V_d$ denote the subgroup of $G$ topologically generated by $H$ and $\gamma^d$ and let $q_d$ denote the homomorphisms of $\KTh$-groups resulting from the application of Prop.~\ref{prop:functoriality of splitting} to the $R[[V_d]]$-$R[[G]]$-bimodule $M=R[[G]]$ and the $R[[H]]$-$R[[H]]$-subbimodule $N=R[[H]]$. Since $M$ and $N$ satisfy condition $(2)$ of Prop.~\ref{prop:functoriality of splitting}, we know that each $q_d$ commutes with $s$.

Now $U_{\tau}\subset V_{d_\tau}$ for each $\tau\in\SQab(G)$ and we may write $q_{\tau}=q'_{\tau}\comp q_{d_\tau}$ with $q'_{\tau}$ corresponding to the $R[[G_\tau]]$-$R[[V_d]]$-bimodule
$$
M=R[[G_\tau]]\ctensor_{R[[U_\tau]]}R[[V_{d_\tau}]]
$$
and the $R[[H_\tau]]$-$R[[H]]$-subbimodule
$$
N=R[[H_\tau]]\ctensor_{R[[U_\tau\cap H]]}R[[H]].
$$
Since $M$, $N$, $\gamma_1=\gamma^{d_{\tau}}$, $\gamma_2=\gamma_{\tau}$, $t_i=\gamma_i-1$ ($i=1,2$) satisfy condition $(1)$ of Prop.~\ref{prop:functoriality of splitting} we conclude that $q'_{\tau}$ commutes with $s$. Likewise, we also see that $q_{\phi}$ commutes with $s$ for each $\phi\in \hat{H}_\tau$. The commutativity of the above diagram follows.
\end{proof}

\begin{rem}
The above corollary also holds for arbitrary pro-finite groups $G=H\rtimes \Gamma$ and for non-noetherian $R$ if one formulates it directly in terms of the $\KTh$-groups $\KTh_1(\cat{TSP}_{R[[H]]}(R[[G]]))$ and $\KTh_0(\cat{TSP}^{R[[H]]}(R[[G]]))$.
\end{rem}

We will now assume that $R=\mathcal{O}$ is a complete discrete valuation ring of characteristic $0$ and residue characteristic $p$. In \cite{CFKSV}, the authors also consider the left denominator set
$$
S^*=\bigcup_{n=0}^{\infty}p^nS.
$$
By the same method as in the proof of Theorem~\ref{thm:splitting abstract}, we obtain the following result.

\begin{thm}
Assume that $\mathcal{O}$ is a complete discrete valuation ring of characteristic $0$ residue characteristic $p$ and that $G$ is a compact $p$-adic Lie group. Then
$$
\KTh_{n+1}(\mathcal{O}[[G]]_{S^*})\isomorph\KTh_{n+1}(\mathcal{O}[[G]][\frac{1}{p}])\oplus \KTh_{n}(\mathcal{O}[[G]][\frac{1}{p}],\mathcal{O}[[G]]_{S^*})
$$
for $n\geq 0$ and
$$
\KTh_0(\mathcal{O}[[G]]_{S^*})\isomorph\KTh_0(\mathcal{O}[[G]][\frac{1}{p}]).
$$
\end{thm}
\begin{proof}
Since $\mathcal{O}[[G]]$ is noetherian, the rings $A=\mathcal{O}[[G]][\frac{1}{p}]$ and $B=\mathcal{O}[[G]]_{S^*}$ are also noetherian. By \cite[Thm.~33.4]{Schneider:LieGroups}  $\mathcal{O}[[G']]$ is Auslander regular for any open $p$-valuable subgroup $G'$ of $G$. We may apply the argument of \cite[Prop.~4.3.4]{FK:CNCIT} to deduce that $A$ is also Auslander regular. Hence, the same is true for $B$. Let $\cat{M}(R)$ denote the abelian category of finitely generated modules over any noetherian ring $R$. Further, we let $\cat{M}(A)_{S^*}$ denote the abelian category of finitely generated, $S^*$-torsion modules over $A$. We then have isomorphisms
\begin{align*}
\KTh_n(A)&\isomorph \KTh_n(\cat{M}(A)),\\
\KTh_n(B)&\isomorph \KTh_n(\cat{M}(B)),\\
\KTh_n(A,B)&\isomorph \KTh_n(\cat{M}(A)_{S^*})
\end{align*}
for all $n\geq 0$. The localisation sequence for Quillen $G$-theory shows that the map $\KTh_0(\cat{M}(A))\mto\KTh_0(\cat{M}(B))$ is a surjection.

Let now $M$ be a module in $\cat{M}(A)_{S^*}$ and fix a system of generators $m_1,\dots,m_n$. Define
$$
M'=\sum_{i=1}^n \mathcal{O}[[G]]m_i.
$$
The $\mathcal{O}[[G]]$-module $M'$ is $p$-torsion-free and $S^*$-torsion, hence $S$-torsion and therefore, finitely generated as $\mathcal{O}[[H]]$-module. From Prop.~\ref{prop:the exact sequence} we obtain the exact sequence of finitely generated $\mathcal{O}[[G]]$-modules
$$
0\mto \mathcal{O}[[G]]\tensor_{\mathcal{O}[[H]]}M'\xrightarrow{\kappa}\mathcal{O}[[G]]\tensor_{\mathcal{O}[[H]]}M'\xrightarrow{\mu}M'\mto 0.
$$
Inverting $p$ we obtain a corresponding exact sequence for $M$ which does not depend on the choice of the generators. By the Waldhausen additivity theorem and by \cite[Thm.~6.8]{Bredon:TopologyAndGeometry} we conclude as in the proof of Theorem~\ref{thm:splitting abstract} that the localisation sequence splits into short split exact sequences.
\end{proof}

The following corollary extends \cite[Prop. 3.4]{CFKSV}.

\begin{cor}
Assume that $G$ is a compact $p$-adic Lie group. Then
$$
\KTh_0(\mathcal{O}[[G]])\mto
\KTh_0(\mathcal{O}[[G]][\frac{1}{p}])=\KTh_0(\mathcal{O}[[G]]_{S^*})
$$
is a split injection. Hence,
$$
\del_0\colon\KTh_1(\mathcal{O}[[G]]_{S^*})\mto\KTh_0(\mathcal{O}[[G]],\mathcal{O}[[G]]_{S^*})
$$
is a surjection.
\end{cor}
\begin{proof}
Let $P$ be an open normal pro-$p$-subgroup of $G$. The kernel of
$\mathcal{O}[[G]]\mto\mathcal{O}[G/P]$ is contained in the Jacobson
radical of $\mathcal{O}[[G]]$ \cite[Prop. 3.2]{Witte:MCVarFF}. Hence,
$\KTh_0(\mathcal{O}[[G]])\mto\KTh_0(\mathcal{O}[G/P])$ is an isomorphism \cite[Thm.~25.3]{Lam:FirstCourseNoncomRings}. Let $K=\mathcal{O}[\frac{1}{p}]$ be the fraction field of $\mathcal{O}$ and consider the commutative diagram
$$
\xymatrix{
\KTh_0(\mathcal{O}[[G]])\ar[r]\ar[d]^{\isomorph}&\KTh_0(\mathcal{O}[[G]][\frac{1}{p}])\ar[d]\\
\KTh_0(\mathcal{O}[G/P])\ar[r]^{e}&\KTh_0(K[G/P])
}
$$
By \cite[Ch.~16, Thm.~34]{Serre:LRFG}, the homomorphism $e$ is split injective.
\end{proof}

The localisation sequence for $\mathcal{O}[[G]]\mto \mathcal{O}[[G]]_{S^*}$ does not split in general. For example, let $G=H\times \Int_p$ with $H$ a finite group. In this case, the kernel of
$$
\KTh_1(\Int_p[[G]])\xrightarrow{i} \KTh_1(\Int_p[[G]][\frac{1}{p}])\subset \KTh_1(\Int_p[[G]]_{S^*})
$$
contains the group $SK_1(\Int_p[H])$, which can be nontrivial if $p$ divides the order of $H$. (In fact, the other inclusion $\ker(i)\subset SK_1(\Int_p[H])$ is also true by \cite[Prop.~5.3]{Witte:NoncommutativeLFunctions}.)

It is tempting to hope that it does split if $G$ has no element of order $p$, but this seems to be difficult to prove even for commutative groups $G$, in which case it is related to the Gersten conjecture. For the first $\KTh$-group however, we can restate the following result obtained by Burns and Venjakob:

\begin{cor}
Let $\mathcal{O}$ be a discrete valuation ring of characteristic $0$ with a residue field $k$ of characteristic $p$. Assume that $G$ is a compact $p$-adic Lie group without $p$-torsion. Then
$$
\KTh_1(\mathcal{O}[[G]]_{S^*})\isomorph\KTh_1(\mathcal{O}[[G]])\oplus\KTh_0(\mathcal{O}[[G]],\mathcal{O}[[G]]_S)\oplus \KTh_0(k[[G]]).
$$
\end{cor}
\begin{proof}
This is \cite[Thm 2.1]{BV:DescentTheory} extended to general complete discrete valuation rings $\mathcal{O}$ of characteristic $0$ and residue characteristic $p$ combined with our new knowledge that $\KTh_1(\mathcal{O}[[G]])$ injects into $\KTh_1(\mathcal{O}[[G]]_S)$.
\end{proof}

\bibliographystyle{amsalpha}
\bibliography{Literature}
\end{document}